\documentclass[reqno,11pt,twoside]{amsart}

\usepackage{xypic} 
\usepackage{graphicx}   
\usepackage{float}
\usepackage{graphics,amssymb}  
\usepackage{amsfonts}
\usepackage{latexsym}
\usepackage{epsf}
\usepackage{mathrsfs}
\usepackage{bbm}
\usepackage{hyperref}
\hypersetup{ 
colorlinks,
citecolor=black,
filecolor=blue,
linkcolor=black,
urlcolor=blue
}
\xyoption{curve}

\makeatletter

\newcommand{\Rmnum}[1]{\expandafter\@slowromancap\romannumeral #1@}
\makeatother

\newtheorem{lemma}{Lemma}[section]
\newtheorem{lem/def}[lemma]{Lemma/Definition}
\newtheorem{proposition}[lemma]{Proposition}
\newtheorem{theorem}[lemma]{Theorem}
\newtheorem{corollary}[lemma]{Corollary}

\theoremstyle{definition}
\newtheorem{example}[lemma]{Example}
\newtheorem{definition}[lemma]{Definition}
\newtheorem{redefinition}[lemma]{Redefinition}
\newtheorem{remark}[lemma]{Remark}

\newcommand{\beq}{\begin{equation*}}
\newcommand{\eeq}{\end{equation*}}
\newcommand{\ba}{\begin{array}}
\newcommand{\ea}{\end{array}}

\newcommand{\Hom}{\mathrm{Hom}}
\newcommand{\RHom}{\mathrm{RHom}}
\newcommand{\Ext}{\mathrm{Ext}}

\newcommand{\End}{\mathrm{End}}
\newcommand{\gr}{\mathrm{gr}}
\newcommand{\ox}{\otimes}

\title[Hochschild cohomology via twisting cochains]{The cup product on Hochschild cohomology via twisting cochains and applications to Koszul rings}
\date{}

\author{Cris Negron}
\address{Department of Mathematics\\Louisiana State University\\
Baton Rouge, LA 70803, USA}
\email{cnegron@lsu.edu}
\thanks{This material is based upon work supported by the National Science Foundation Graduate Research Fellowship under Grant No. DGE-1256082}

\begin{document}
\maketitle

\begin{abstract}
Given an acyclic twisting cochain $\pi:C\to A$, from a curved dg coalgebra $C$ to a dg algebra $A$, we show that the associated twisted hom complex $\Hom^\pi_k(C,A)$ has cohomology equal to the Hochschild cohomology of $A$, as a graded ring.  As a corollary we find that the Hochschild cohomology of a Koszul algebra $A$, along with its cup product, is a subquotient of the tensor product algebra $A^!\ox A$ of $A$ with its Koszul dual.
\end{abstract}

\section{Introduction}

Let $k$ be a field of arbitrary characteristic.  The Hochschild cohomology of an dg ($k$-)algebra $A$ can be defined as the graded group of bimodule extensions
\[
HH^\bullet(A)=\Ext_{A\text{-}\mathrm{bimod}}^\bullet(A,A).
\]
This cohomology was popularized in the many works of Gerstenhaber (see e.g. \cite{gerstenhaber63,gerstenhaber64}).  Hochschild cohomology is fundamental in that it is the cohomology associated to a dg algebra $A$ which ``controls the formal deformation theory of $A$", in the general sense of deformation theory via dg Lie algebras \cite{manetti09,KSdeformation} (see also the closing comments of Section~\ref{rHH}).
\par

As with any group of self extensions, Hochschild cohomology carries a natural Yoneda product, which is often referred to as the cup product.  Hochschild cohomology, along with the cup product, has been shown to have very strong relations to loop space (co)homology in topological settings~\cite{menichi01,felix04}, and is known to be a derived invariant~\cite{keller03}.
\par

In this paper we give a general result relating twisting cochains to the cup product on Hochschild cohomology, then discuss its specific appearance for Koszul algebras.

\begin{theorem}[\ref{thm1}]\label{thm:-2}
Let $A$ be any dg algebra, $C$ be any (curved) dg coalgebra, and $\pi:C\to A$ be an {\it acyclic twisting cochain}.  Let $\Hom^\pi_k(C,A)$ denote the associated {\it twisted hom complex}.  There is a dg algebra quasi-isomorphism
\[
\Hom_k^\pi(C,A)\overset{\sim}\to \RHom_{A\text{-}\mathrm{bimod}}(A,A)
\]
and subsequent identification of graded rings $H^\bullet(\Hom_k^\pi(C,A))=HH^\bullet(A)$.
\end{theorem}

In the text $\RHom_{A\text{-}\mathrm{bimod}}(A,A)$ will be represented by the endomorphism dg algebra of a particular semi-projective approximation of $A$.  A twisting cochain is a degree $1$ $k$-linear map $\pi$ from $C$ to $A$ which satisfies the Maurer-Cartan equation, and the corresponding twisted hom complex is the standard set of graded homs
\[
\Hom^\pi_k(C,A)=\oplus_i\left\{\ba{c}\text{homogeneous degree $i$ $k$-linear maps}\\ f:C\to A\ea\right\}
\]
with the convolution product and altered differential
\[
d_{\Hom^\pi}(f)=d_A f-(-1)^{|f|}fd_C-[\pi,f].
\]
This gives $\Hom^\pi_k(C,A)$ the structure of a dg algebra.  
\par

Twisting cochains have primarily been of use in topological and geometric studies \cite{brown59,chen77,gugenheim82}, although they have proved fundamental in recent algebraic inquiries into $A_\infty$-algebras initiated by Lef\'{e}vre-Hasegawa's thesis \cite{L} and the subsequent works of Keller \cite{K,K1,K2,K3}.  In the topological/geometric settings the object $A$ is usually a dg algebra with nonvanishing differential (e.g. the singular chains of a topological group), while in the ring theoretic context $A$ will usually be concentrated in degree $0$.
\par

Although the cup product on Hochschild cohomology is of independent interest (see e.g. the support variety theory of~\cite{SS04} and the applications to loop space (co)homology~\cite{menichi01,felix04}), we are also interested in it for its value as a derived invariant and for the assisting role it plays in computations of the graded Lie structure on Hochschild cohomology (see e.g.~\cite{grimleynguyenwitherspoon}).\footnote{One can see~\cite{gerstenhaber64,KSdeformation} for more information on this Lie structure.}  We also believe that the relationship established in Theorem~\ref{thm:-2} is somewhat fundamental.  Indeed, we are suggesting that dg coalgebras with acyclic twisting cochains can act as replacements for the (generally) larger structures of the bar construction, bar resolution, and Hochschild cochain complex.  One can see~\cite{negronwitherspoon15} for another application of this philosophy.

As noted above, Theorem~\ref{thm:-2} has a specific life for Koszul algebras.  While we'll avoid giving a formal definition of Koszul algebras here, some important examples include quantum polynomial rings, Sklyanin algebras, Universal enveloping algebras, Clifford algebras, Steenrod algebras, and Weyl algebras (see~\cite{Pr}).  For a {\it graded} Koszul algebra, for example, we will have

\begin{theorem}[\ref{thm}]\label{thm:-1}
Let $A=k\langle x_1,\dots, x_n\rangle/(R)$ be a Koszul algebra (with $R$ the $k$-space of homogeneous degree $2$ relations) and $A^!=k\langle\lambda_1,\dots,\lambda_n\rangle/(R^\perp)$ be its Koszul dual.  Let $e$ be the degree $1$ element $e=\sum_i \lambda_i\ox x_i$ in the tensor product algebra $A^!\ox A$.  Then
\begin{enumerate}
\item the graded commutator operation $[e,-]:A^!\ox A\to A^!\ox A$ is a square zero degree $1$ derivation on the algebra $A^!\ox A$, where we grade by the degree on $A^!$,
\item the cohomology of the resulting dg algebra admits an identification of graded rings $H^\bullet(A^!\ox A)=HH^\bullet(A)$.
\end{enumerate}
\end{theorem}

In the statement of the above theorem the $\lambda_i$ are dual to $x_i$ so that $k\{\lambda_1,\dots,\lambda_n\}=(k\{x_1,\dots,x_n\})^\ast$.  When $A$ is filtered (not graded) we need to replace $A^!$ with a certain curved dg algebra which plays the analogous role.
\par

We should mention here the work of Buchweitz, Green, Snashall, and Solberg \cite{BS}.  In the paper \cite{BS} the authors relate the multiplication on the Hochschild cohomology of a Koszul (path) algebra $A$ to a comultiplicative structure on a special collection of right ideals in $A$.  One manner in which our results differ is that the methods presented here are essentially basis free--the element $e$ does not depend on choice of basis--while those of \cite{BS} seem to depend heavily on choices of bases and a concrete analysis of a certain coalgebra related to the Koszul ring.
\par

The work presented here also appears, in an alternate form with an additional computation of the Hochschild cohomology ring of the universal enveloping algebra of the Heisenberg Lie algebra, in the first chapter of my dissertation~\cite{negronthesis}.

\subsection{Organization}

Section~\ref{dg} is dedicated to establishing the necessary background information, and in Section~\ref{Tw} we define twisting cochains and some related constructions.  Our presentation of the twisted tensor product seems to be novel, and allows for a clear proof of Theorem~\ref{thm:-2}.  In Section~\ref{cplXez} we give a fundamental relationship between twisted tensor products and twisted homs.  Sections~\ref{cupprod} and~\ref{rHH} are dedicated to the presentation of the main theorem and its proof, and Sections~\ref{sec:koszul} and \ref{kos} are dedicated to an analysis of Koszul algebras.

\subsection*{Acknowledgments}

Thanks to thank Xingting Wang, Guangbin Zhuang, and James Zhang for their support and assistance in this research.  Thanks also to James Zhang for pushing me to think more deeply about the basics of twisting cochains. 

\section*{Notations and conventions}
\label{nots}
 
 Let $R$ be an arbitrary ring.  By a ``$R$-module" we mean a \emph{left} $R$-module unless stated otherwise.  We will always use the cohomological indexing convention
\[
X=\cdots\overset{d}\to X^{n-1}\overset{d}\to X^n\overset{d}\to X^{n+1}\overset{d}\to \cdots,
\]
and do not distinguish between chain complexes and cochain complexes.  Given $R$-complexes $X$ and $Y$ we write $\Hom_R(X,Y)$ for the standard Hom complex
\[
\Hom_R(X,Y)=\bigoplus_{n\in\mathbb{Z}}(\prod_i\Hom_R(X^i, Y^{i+n}))
\]
For any homogenous function $\theta\in\Hom_R(X,Y)$ of degree $n$, the differential $d$ is given by the formula $d(\theta)=d_Y\theta-(-1)^{n}\theta d_X$.  For a $k$-complex $X$ we let $X^\ast$ denote the chain dual
\[
X^\ast=\Hom_k(X,k)=\cdots\to (X^{i+1})^\ast\to (X^i)^\ast\to (X^{i-1})^\ast\to \cdots
\]
\par

Sweedler's notation will be used to denote the comultiplication on a coalgebra $C$.  So the element $\Delta(c)$ will be written $\Delta(c)=c_1\ox c_2$, with the sum implicit.  To say this more clearly, ``$c_1\ox c_2$'' is simply shorthand for some expression of the element
\[
\Delta(c)=\sum_i c_{i_1}\ox c_{i_2}
\]
in the tensor product $C\ox C$.  Higher iterations of the comultiplication will be denoted using similar notation.  For example, the element 
\[
(\Delta\ox id)\Delta(c)=(id\ox\Delta)\Delta(c)
\]
will be denoted $c_1\ox c_2\ox c_3$.  Again, there is an implicit sum.  If $C$ is graded, and $c\in C$ is homogeneous, then the $c_1, c_2,$ etc. will always be taken to be homogeneous. 
 
\section{Preliminaries on (curved) dg algebras and (curved) dg coalgebras}
\label{dg}

\subsection{Dg algebras and coalgebras}

Recall that a dg algebra is a chain complex $(A, d)$ equipped with a unit $k\to A$ and associative multiplication $\mu:A\ox A\to A$ which are both chain maps.  A dg coalgebra is defined dually to be a complex $(C, d)$ with a coalgebra structure such that each structure map $C\to k$, $\Delta:C\to C\ox C$, is a chain map.  We will call a dg (co)algebra locally finite if it is finite dimensional in each homological degree.  A dg algebra $A$ (resp. dg coalgebra $C$) is said to be augmented (resp. coaugmented) if it comes equipped with a dg map $A\overset{\epsilon}\to k$ (resp. $k\overset{u}\to C$).
\par
Given an arbitrary dg algebra $A$ and dg coalgebra $C$ the hom complex $\Hom_k(C,A)$ becomes a dg algebra under the convolution product
\[
f\ast g:=\mu_A(f\ox g)\Delta_C:c\mapsto (-1)^{|c_1||g|}f(c_1)g(c_2).
\]
In particular the dual $C^\ast=\Hom_k(C,k)$ is a dg algebra.  One can check that the dual $A^\ast=\Hom_k(A,k)$ of any locally finite dg algebra, which is additionally bounded above or below, is a dg coalgebra under the coproduct $\Delta(\gamma)=\gamma\mu$.  The double dual of a locally finite dg (co)algebra $A$, which is bounded above or below, is naturally isomorphic to $A$ via the standard map
\begin{equation}\label{evelmap}
\ba{c}
ev:A\to (A^\ast)^\ast\\
a\mapsto (\phi\mapsto (-1)^{|a||\phi|}\phi(a)).
\ea
\end{equation}

The tensor product of dg (co)algebras is again a dg (co)algebra under the differential $d_{A\ox A'}=d_{A}\ox id_{A'}+id_{A}\ox d_{A'}$, and we can define the opposite dg algebra $A^{op}$ to be the complex $A$ with the opposite multiplication $a\cdot^{op}b:=(-1)^{|a||b|}ba$.
\par

Given a dg algebra $A$, a $k$-complex $M$ is called a left (resp. right) dg module over $A$ if it is a graded $A$-module, after we forget the differential, and the action map
\[
A\ox M\to M\ \ (\text{resp.}\ M\ox A\to M)
\]
is a map of chain complexes.  Similarly, $M$ is a dg bimodule if it is a graded bimodule over $A$ and the action map $A\ox M\ox A\to M$ is one of chain complexes.  A bimodule over a dg algebra $A$ can, as in the non-dg case, also be seen as a module over the enveloping algebra $A^e=A\ox A^{op}$.  
\par

Given dg $A$-modules $M$ and $N$, we define the hom complex $\Hom_{A}(M,N)$ in the usual way.  That is
\[
\Hom_{A}(M,N)=(\oplus_i \Hom_{A}^i(M,N), d_{\Hom})
\]
with the usual differential $d_{\Hom}:f\mapsto d_N f-(-1)^{|f|}f d_M$.  Here the notation $\Hom_{A}^i(M,N)$ denotes the set of homogenous degree $i$ maps $f:M\to N$ satisfying
\[
f(am)=(-1)^{|f||a|}a f(m)
\]
for any homogeneous $a,b\in A$.  Taking $A=B^e$ for a dg algebra $B$ gives the appropriate definition of the hom complex for bimodules $\Hom_{B\text{-bimod}}(M,N)=\Hom_{B^e}(M,N)$.

\subsection{Curved dg algebras and coalgebras}
\label{sec:curved_stuffs}

We introduce the notion of a curved dg (co)algebra, following \cite{P}.  These structures will be needed in our account of Koszul duality for nonaugmented algebras, such as Weyl algebras and Clifford algebras.

\begin{definition}[Curved dg (co)algebras]
A curved dg algebra is a graded algebra $A=\oplus_{i\in\mathbb{Z}}A^i$ along with a degree $1$ graded derivation $d_A$, and a degree $2$ element $c_A\in A^2$, so that
\[
d_A^2=[c_A,-]\ \ \mathrm{and}\ \ d_A(c_A)=0.
\]
(Here $d_A^2$ is the square $d_Ad_A$.)  Dually, a curved dg coalgebra is a graded coalgebra $C=\oplus_{i\in\mathbb{Z}}C^i$ along with a degree $1$ coderivation $d_C$, and degree $2$ function $f_C:C\to k$ satisfying
\[
d_C^2=(f_C\ox id-id\ox f_C)\Delta\ \ \mathrm{and}\ \ f_Cd_C=0.
\]
We may denote a curved dg algebra (resp. coalgebra) as a triple $(A,d_A,c_A)$ (resp. $(C,d_C,f_C)$).
\end{definition}

As with dg algebras and coalgebras, we have some standard constructions.  Given a curved dg algebra $A$ and a curved dg coalgebra $C$ the set of graded maps $\Hom_k(C,A)=\bigoplus_n\big(\prod_i\Hom_k(C^i,A^{i+n})\big)$ becomes a curved dg algebra under the convolution product, standard derivation $d(\xi)=d_A\xi-(-1)^{|\xi|}\xi d_C$, and curvature
\[
c_{\Hom}=c_A\epsilon_C-1_Af_C
\]
\cite[Section 6.2]{positselski11}.  In particular, the graded dual $C^\ast$ of any curved dg coalgebra becomes a curved dg algebra with curvature element $c_{C^\ast}=-f_C$.  The graded dual of any locally finite curved dg algebra $A$ becomes a curved dg coalgebra with the obvious coproduct, derivation $d(\eta)=-(-1)^{|\eta|}\eta d_A$, and curvature function $f_{A^\ast}=-ev_{c_A}$.  When $C$ is locally finite, the evaluation map $ev:C\to {C^\ast}^\ast$ provides an isomorphism of curved dg coalgebras between $C$ and its double dual.  Finally, the tensor product $A\ox A'$ of curved dg algebras will again be a curved dg algebra with $d_{A\ox A'}=d_A\ox id_{A'}+id_A\ox d_{A'}$ and $c_{A\ox A'}=c_A\ox 1+1\ox c_{A'}$.
\par

Theoretically, curved dg structures arise as deformations of dg algebras.  For example, a cocycle in the second Hochschild cohomology of a dg algebra $A$ will correspond to a curved dg $k[t]/(t^2)$-algebra, or more generally curved $A_\infty$ $k[t]/(t^2)$-algebra, which reduces to $A$ at $t=0$.

\section{Twisting cochains}
\label{Tw}

We first give the definition in the non-curved setting, then address the curved situation independently.  The following definition is standard and can be found, for example, in \cite{gugenheim_etal90} or \cite{LV}.

\begin{definition}[Twisting cochain]
\label{twc}
Let $C$ be a dg coalgebra and $A$ be a dg algebra.  A degree $1$ linear map $\pi: C\to A$ is called a twisting cochain if $\pi$ satisfies the equation 
\begin{equation}\label{MC}
-(d_A \pi+\pi d_C)+\mu (\pi\ox \pi)\Delta=0.
\end{equation}
\end{definition}

In other sources, the formula in (\ref{MC}) may appears as 
\[
d_A \pi+\pi d_C+\mu (\pi\ox \pi)\Delta=0.
\]
One can mediate between the two perspectives by replacing $\pi$ with $-\pi$.  Assuming $k$ is of characteristic $\neq 2$, this alternate form of (\ref{MC}) is exactly the statement that $\pi$ is a solution to the Maurer-Cartan equation
\[
d(\pi)+\frac{1}{2}[\pi,\pi]=0,
\]
where $[,]$ denotes the graded commutator on the dg algebra $\Hom_k(C,A)$.

\begin{example}\label{exampleLie}
Let $\mathfrak{g}$ be a Lie algebra and $k[\Sigma \mathfrak{g}]$ be the free {\it graded} commutative algebra generated by the degree $-1$ $k$-space $\Sigma \mathfrak{g}$.  Given $x\in \mathfrak{g}$ we let $X$ denote the corresponding element in $\Sigma \mathfrak{g}$ with shifted degree.  We extend the operation 
\[
\Delta:\Sigma \mathfrak{g}\to k[\Sigma \mathfrak{g}]\ox k[\Sigma \mathfrak{g}],\ \ X\mapsto 1\ox X+X\ox 1
\]
(uniquely) to an algebra map from $\Delta:k[\Sigma \mathfrak{g}]\to k[\Sigma \mathfrak{g}]\ox k[\Sigma \mathfrak{g}]$, which gives $k[\Sigma \mathfrak{g}]$ a graded bialgebra structure.
\par

The linear map $k[\Sigma \mathfrak{g}]^{-2}\to k[\Sigma \mathfrak{g}]^{-1}=\Sigma \mathfrak{g}$ which takes a monomial $XY$ to the Lie bracket $-[x,y]$, with shifted degree, extends to a unique {\it co}derivation $d_\mathfrak{g}$ on $k[\Sigma \mathfrak{g}]$.  We then get a canonical twisting cochain $\pi: k[\Sigma \mathfrak{g}]\to U(\mathfrak{g})$ given in degree $-1$ by $X\mapsto x$.  Indeed, on degree $-2$ elements we have
\[
\ba{rl}
\Delta(XY)=\Delta(X)\Delta(Y)&=XY\ox 1+X\ox Y+(-1)^{|X||Y|}Y\ox X+1\ox XY\\
&=XY\ox 1+X\ox Y-Y\ox X+1\ox XY
\ea
\]
and so
\[
\ba{rl}
\pi\ast\pi (XY)&=\pi(XY)\pi(1)-\pi(X)\pi(Y)+\pi(Y)\pi(X)+\pi(1)\pi(XY)\\
&=-[x,y]=\pi d_\mathfrak{g}(XY).
\ea
\]
This is the twisting cochain condition.  This twisting cochain is an example of the canonical twisting cochain of Lemma \ref{lemdef}.
\end{example}

\begin{definition}[Twisted homs]
Given a twisting cochain $\pi:C\to A$, we define the dg algebra of twisted homs $\Hom_k^\pi(C,A)$ as the space of graded homs $\Hom_k^\pi(C,A)=\oplus_i \Hom_k^i(C,A)$ along with convolution product $\ast$ and differential 
\[
\ba{rl}
d_{\Hom^\pi_k(C,A)}(f)&:=d_Af-(-1)^{|f|}fd_C-(\pi\ast f-(-1)^{|f|}f\ast \pi)\\
&=d_{\Hom_k(C,A)}(f)-[\pi,f].
\ea
\]
\end{definition}

One can easily verify that $\Hom_k^\pi(C,A)$ is a in fact a dg algebra, or simply see \cite[Proposition 2.1.6]{LV}.  We also have the analogous definition of the twisted homs $\Hom_k^\pi(C,M)$, where $M$ is a dg $A$-bimodule.

\begin{lem/def}\label{460}
Suppose $\pi:C\to A$ is a twisting cochain.  There is a functor
\[
(-)^\pi:\mathrm{dg\ \Hom_k(C,A)\text{-}bimod}\to \mathrm{dg\ \Hom_k^{\pi}(C,A)\text{-}bimod}.
\]
This functor takes a dg bimodule $(M,d_M)$ to the bimodule $(M^\pi,d^\pi_M)$ which is $M$ as a graded space, has the $\Hom^\pi_k(C,A)$-action induced by the algebra identification $\Hom_k(C,A)=\Hom^\pi_k(C,A)$, and differential $d^\pi_M:=d_M-[\pi,-]$.  For any $\varphi:M\to N$ we simply take $\varphi^\pi:=\varphi$.
\end{lem/def}

\begin{proof}
Take $d_H=d_{\Hom_k(C,A)}$ and $d^\pi=d_{\Hom_k^\pi(C,A)}$.  For any $m\in M$ we have
\begin{equation}\label{470}
\ba{rl}
(d^\pi_M)^2(m)&=d_M^2(m)-d_M([\pi,m])-[\pi,d_M(m)]+[\pi,[\pi,m]]\\
&=-[d_H(\pi),m]+[\pi,d_M(m)]-[\pi,d_M(m)]+[\pi,[\pi,m]]\\
&=-[d_H(\pi),m]+[\pi,[\pi,m]].
\ea
\end{equation}
But
\[
[\pi,[\pi,m]]=\pi^2m-m\pi^2-(-1)^{|m|}\pi m \pi-(-1)^{|m|+1}\pi m \pi=[\pi^2,m]
\]
and $\pi^2=d_H(\pi)$ by the twisting cochain condition.  So the final equation of (\ref{470}) vanishes and we get that $d_M^\pi$ is in fact a differential on $M$.
\par

We need to check now that the equation $d_M^\pi(fm)=d^\pi(f)m+(-1)^{|f|}fd_M^\pi(m)$ for $f\in \Hom^\pi_k(C,A)$ and $m\in M$.  But this is clear since $d_M^\pi=d_M+[\pi,-]$ and $d_M$ and $[\pi,-]$ satisfy the equations
\[
d_M(fm)=d_H(f)m+(-1)^{|f|} f d_M(m)\ \ \text{and}\ \ [\pi,fm]=[\pi,f]m+(-1)^{|f|}f[\pi,m].
\]
One similarly verifies the equation for $d_M^\pi(mf)$.  Finally, one sees that for any $\varphi:M\to N$, $\varphi^\pi=\varphi$ is still a dg map since for any $m\in M$
\[
\varphi(d_M^\pi(m))=\varphi(d_M(m))+\varphi([\pi,m])=d_M(\varphi(m))+[\pi,\varphi(m)]=d_M^\pi(\varphi(m)).
\]
\end{proof}

\begin{lemma}\label{494}
For any dg coalgebra $C$ and dg algebra $A$, the tensor complex $A\ox C\ox A$ is a dg $\Hom(C,A)$-bimodule under the left and right actions
\[
f\cdot (a\ox c\ox b):=(-1)^{|f|(|a|+|c_1|)}a\ox c_1\ox f(c_2)b
\]
and
\[
(a\ox c\ox b)\cdot f:=(-1)^{|f|(|c|+|b|)}af(c_1)\ox c_2\ox b,
\]
for $a,b\in A$, $c\in C$.
\end{lemma}

\begin{proof}
The verification of this fact is a sequence of tedious but straightforward calculations.  We only check compatibility with the differential under the left action.  Take $f$ and $a\ox c\ox b$ as above.  Then we have
\[
\ba{l}
d(f\cdot(a\ox c\ox b))\\
=\pm d(a)\ox c_1\ox f(c_2)b\pm a\ox d(c_1)\ox f(c_2)b\pm a\ox c_1\ox d(f(c_2))b\\
\hspace{4mm}\pm a\ox c_1\ox f(c_2)d(b)\\
=\pm d(a)\ox c_1\ox f(c_2)b\pm a\ox d(c_1)\ox f(c_2)b\pm a\ox c_1\ox f(c_2)d(b)\\
\hspace{5mm}\pm a\ox c_1\ox d_{\Hom}(f)(c_2)b+(-1)^{(|f|+1)(|a|+|c_1|)+|f|} a\ox c_1\ox f(d(c_2))b
\ea
\]
where all the signs $\pm$ are the appropriate Koszul signs, and all the $d$ must be given the appropriate subscript.  Using the fact that $d_C$ is a coderivation, this final equation can then be rewritten
\[
(-1)^{|f|}f\cdot d_{A\ox C\ox A}(a\ox c\ox b)+d_{\Hom}(f)\cdot (a\ox c\ox b).
\]
The verification of the formula
\[
\ba{c}
d_{A\ox C\ox A}((a\ox c\ox b)\cdot f)\\
=d_{A\ox C\ox A}(a\ox c\ox b)\cdot f+(-1)^{|f|(|a|+|c|+|b|)}(a\ox c\ox b)\cdot d_{\Hom}(f)
\ea
\]
is similar.
\end{proof}

To gain some understanding of these actions we may consider the two extreme cases, when $C=k$ and when $A=k$.  When $C=k$ the actions reduce to the inner actions of $\Hom_k(k,A)=A$ on $A\ox A$.  When $A=k$, the actions of the dual $\Hom_k(C,k)=C^\ast$ on $C$ are the standard left and right actions corresponding to the respective right and left {\it co}actions of $C$ on itself \cite{M}.  When $C$ and $A$ have nonvanishing (co)augmentation (co)ideals, we are integrating the inner action of $A$ with the standard actions of $C^\ast$ to produce a natural action of the convolution algebra $\Hom_k(C,A)$ on the tensor complex.

\begin{definition}[Twisted tensor products]
Given a twisting cochain $\pi:C\to A$, we define the twisted tensor product $A\ox_\pi C\ox_\pi A$ as the value of the functor $(-)^{\pi}$ on the bimodule $A\ox C\ox A$ (with right and left actions as in described in Lemma \ref{494}).  In other words,
\[
A\ox_\pi C\ox_\pi A=(A\ox C\ox A)^\pi=(A\ox C\ox A, d_{A\ox C\ox A}-[\pi,-]).
\]
\end{definition}

The fact that the twisted tensor product, according to Lemma \ref{460}, is a dg bimodule over the twisted hom complex $\Hom^\pi_k(C,A)$ will be an essential point in our proof of the main result Theorem \ref{thm1}.
\par

On elements, the differential on the twisted tensor product will be given by the formula
\[
\ba{l}
d_{A\ox C\ox A}^\pi(a\ox c\ox b)\\
=d_{A\ox C\ox A}(a\ox c\ox b)+(-1)^{|a|+|c_1|}a\ox c_1\ox \pi(c_2)b\\
\hspace{1cm}-(-1)^{|a|+|c|+|b|+|b|+|c|}a\pi(c_1)\ox c_2\ox b\\
=d_{A\ox C\ox A}(a\ox c\ox b)+(-1)^{|a|+|c_1|}a\ox c_1\ox \pi(c_2)b-(-1)^{|a|}a\pi(c_1)\ox c_2\ox b.
\ea
\]
So we can write the differential on the twisted product in the more conventional form
\[
d_{A\ox C\ox A}+\left(\mu(id_A\ox\pi)\ox id_C\ox id_A-id_A\ox id_C\ox\mu(\pi\ox id)\right)\left(id_A\ox\Delta\ox id_A\right)
\]
\cite{husemoller74}.  The twisted tensor product $A\ox_\pi C\ox_\pi A$ will always have a canonical dg $A$-bimodule structure under the outer $A$-actions, and we will view the twisted tensor product as an object in dg $A$-bimod in what follows.
\par

As an example, the twisted tensor product $U(\mathfrak{g})\ox_\pi k[\Sigma\mathfrak{g}]\ox_\pi U(\mathfrak{g})$ associated to the twisting cochain of Example \ref{exampleLie} will recover the standard Koszul resolution of \cite{Pr}.  For the familiar reader, we also note that one recovers the standard bar resolution as the twisted tensor product of a dg algebra $A$ with its bar construction via the universal twisting cochain.

\subsection{Twisting cochains in the curved setting}

Here we follow \cite[Section 6.2]{positselski11}.  The reader may also refer to \cite{nicolas08}.

\begin{definition}[Twisting cochains with curvature]
A degree $1$ linear map $\pi:C\to A$ from a curved dg coalgebra to a curved dg algebra is called a twisting cochain if the equation
\[
(c_A\epsilon_C-1_Af_C)-(d_A\pi+\pi d_C)+\mu(\pi\ox\pi)\Delta=0
\]
holds.
\end{definition}

Here, again, we differ from some other references by a sign.  We will only be interested in the case in which $A$ is a dg algebra.  In this setting we still get a twisted tensor products $A\ox_\pi C\ox_\pi A$.  The differential on this complex is, oddly enough, given by the same formula as in the non-curved setting:
\begin{equation}
\label{d900}
d_{A\ox C\ox A}+\left(\mu(id_A\ox\pi)\ox id_C\ox id_A-id_A\ox id_C\ox\mu(\pi\ox id)\right)\left(id_A\ox\Delta\ox id_A\right).
\end{equation}
\par

As one can see from the above formula, the curvature disappears at this level.  Indeed, for the remainder of the paper we will be able to provide a uniform analysis of the curved and non-curved situations.  We outline below the manner in which one arrives at the above formula for the twisted tensor product in the curved situation.

\begin{definition}[Curved bimodules]
Suppose we have a curved dg algebra $A=(A,d_A,c_A)$. An $A$-bimodule $M$ is called a curved bimodule if $M$ comes equipped with a grading $M=\oplus_i M^i$ and a degree $1$ operation $d_M$ satisfying
\[
d_M(am)=d_A(a)m+(-1)^{|a|}ad_M(m)\ \ \text{and}\ \ d_M(ma)=d_M(m)a+(-1)^{|m|}md_A(a)
\]
and $d_M^2=[c_A,-]$.  A morphism of curved bimodules is a graded $A$-bimodule map $\varphi:M\to N$ which satisfies $d_N\varphi=\varphi d_M$.
\end{definition}

\begin{lem/def}\label{590}
For any twisting cochain $\pi:C\to A$ from a curved dg coalgebra to a curved dg algebra the complex
\[
\Hom^\pi_k(C,A)=(\Hom_k(C,A),d_{\Hom}-[\pi,-])
\]
is a dg algebra.
\end{lem/def}

\begin{proof}
Take $d^\pi=d_{\Hom}-[\pi,-]$.  Since the sum of algebra derivations is again an algebra derivation the operation $d^\pi$ is an algebra derivation.  We need only check that it is square $0$.  For homogeneous $g\in \Hom_k(C,A)$ we get
\[
\ba{rl}
(d^{\pi})^2(g)&=d_{\Hom}^2(g)-d_{\Hom}([\pi,g])-[\pi,d_{\Hom}(g)]+[\pi[\pi,g]]\\
&=[c_A\epsilon,g]-[1_Af_C,g]-[d_{\Hom}(\pi),g]+[\pi,d_{\Hom}(g)]\\
&\hspace{4mm}-[\pi,d_{\Hom}(g)]+[\pi^2,g]\\
&=[c_A\epsilon,g]-[1_Af_C,g]-[d_{\Hom}(\pi),g]+[\pi^2,g]\\
&=[c_A\epsilon-1_Af_C-d_{\Hom}(\pi)+\pi^2,g]\\
&=0,
\ea
\]
since $\pi$ is a twisting cochain and therefore $c_A\epsilon-1_Af_C-d_{\Hom}(\pi)+\pi^2=0$.
\end{proof}

\begin{lem/def}\label{lem613}
Given a twisting cochain $\pi:C\to A$ from a curved dg coalgebra to a curved dg algebra, we get a functor
\[
(-)^{\pi}:\mathrm{curved\ \Hom_k(C,A)\text{-}bimod}\to \mathrm{dg\ \Hom_k^\pi(C,A)\text{-}bimod}.
\]
On objects, for a curved bimodule $M$ we take $M^\pi=M$ as a graded bimodule over $\Hom_k^\pi(C,A)=\Hom_k(C,A)$, and give $M^\pi$ the differential $d_M^\pi:=d_M-[\pi,-]$.  Given $\varphi:M\to N$ we take $\varphi^\pi=\varphi$.
\end{lem/def}

Recall that the natural curvature on the convolution algebra $\Hom_k(C,A)$ is the function $c_A\epsilon-1_Af_C$. 

\begin{proof}
Checking that $d_M^2=0$ is formally similar to the computation given for Lemma \ref{590}.  The remainder of the proof is exactly the same as that of Lemma \ref{460}.
\end{proof}

\begin{lemma}\label{lem627}
For a (non-curved) dg algebra $A$ and curved dg coalgebra $C$, the tensor complex $A\ox C\ox A$ is a curved bimodule over $\Hom_k(C,A)$ under the same actions as in Lemma \ref{494}.
\end{lemma}

\begin{proof}
Save for compatibility with the curvature, this is the same as Lemma \ref{494}.  For compatibility with the curvature we have
\begin{equation}\label{543}
\ba{rl}
d^2(a\ox c\ox b)&=a\ox d_C^2(c)\ox b\\
&=a\ox f_C(c_1)c_2\ox b-a\ox c_1 f_C(c_2)\ox b\\
&=a\ox c_1 \ox (-1_Af_C(c_2)) b-a(-1_Af_C(c_1))\ox c_2\ox b\\
&=[c_{\Hom},a\ox c\ox b].
\ea
\end{equation}
\end{proof}

We then get the twisted tensor product, again, as the value of $(-)^{\pi}$ on this bimodule
\[
A\ox_\pi C\ox_\pi A=(A\ox C\ox A)^\pi=(A\ox C\ox A,d_{A\ox C\ox A}-[\pi,-]).
\]
The twisted tensor product will be viewed as a dg $A$-bimodule under the outer $A$-actions.

\begin{remark}
The proof of Lemma \ref{lem627} breaks if we allow $A$ to be curved.  Specifically, the sequence of equalities (\ref{543}) will not hold.
\end{remark}

\section{Maps from the twisted tensor complex as twisted homs}
\label{cplXez}

For the remainder of the paper by a ``twisting cochain" $\pi:C\to A$ we will mean a twisting cochain from a, possibly curved, dg coalgebra to a dg algebra.  One should recall our definition of the hom complex $\Hom_{A^e}(M,N)$ from Section \ref{dg}, for dg bimodules $M$ and $N$.

\begin{proposition}\label{prop731}
Suppose $\pi:C\to A$ is a twisting cochain, and that $M$ is any dg $A$-bimodule.  Then the restriction map
\[
\mathrm{res}_M:\Hom_{A^e}(A\ox_\pi C\ox_\pi A, M)\to \Hom_k^\pi(C,M)
\]
is an isomorphism of chain complexes.  These restrictions together produce a natural isomorphism $\mathrm{res}:\Hom_{A^e}(A\ox_\pi C\ox_\pi A,-)\to \Hom_k^\pi(C,-)$.
\end{proposition}

\begin{proof}
Take $d$ to be the differential on $\Hom_{A^e}(A\ox_\pi C\ox_\pi A, M)$ and $d'$ to be the differential on $\Hom_k^\pi(C,M)$.  We need to check the formula
\[
d(f)|C=d_{\Hom}(f|C)-(\pi\ast (f|C)-(-1)^{|f|}(f|C)\ast\pi)
\]
for any homogenous $A^e$-linear map $f:A\ox_\pi C\ox_\pi A\to M$.  We proceed directly.  Take $c\in C$.  Then
\[
\ba{l}
d(f)(c)\\
=d_Mf(c)-(-1)^{|f|}(f d_{A\ox C\ox A}(c)+f(\pi(c_1)\ox c_2\ox 1-(-1)^{|c_1|}1\ox c_1\ox \pi(c_2))\\
=d_Mf(c)-(-1)^{|f|}fd_{C}(c)-(-1)^{|f|+|f|(|c_1|+1)}\pi(c_1)f(c_2)+(-1)^{|c_1|+|f|}f(c_1)\pi(c_2)\\
=d_{\Hom_k(C,M)}(f|C)(c)-((\pi\ast f)(c)-(-1)^{|f|}(f\ast\pi)(c))\\
=d_{\Hom_k(C,M)}(f|C)(c)-[\pi, f](c)\\
=d'(f|C)(c).
\ea
\]
It is obvious that restriction is natural in $M$, so the second statement is immediate.
\end{proof}

We will call a pair $(C,M)$, with $C$ a curved dg coalgebra and $A$ a dg algebra, {\it sufficiently finite} if the natural dg algebra embedding
\begin{equation}\label{naturalmapper}
C^\ast\ox A\to \Hom_k(C,A),\ \ f\ox a\mapsto (c\mapsto (-1)^{|c||a|}f(c)a)
\end{equation}
is an isomorphism.  For any dg $A$-bimodule $M$, we say the pair $(C,M)$ is sufficiently finite if $(C,A)$ is sufficiently finite and the map analogous to (\ref{naturalmapper}), with $A$ replaced by $M$, is an isomorphism.  Some easy examples of sufficiently finite pairs $(C,A)$ would be when $C$ is finite dimensional, or when $A$ is finite dimensional, or when $C$ is locally finite and $A$ is bounded above and below, etc.

\begin{definition}[The functor $C^\ast\ox^\pi-$]\label{defsup_pi}
Let $\pi:C\to A$ be a twisting cochain on a sufficiently finite pair $(C,A)$.  Then we define the functor
\[
C^\ast\ox^\pi-:\text{dg }A\text{-bimod}\to \text{dg }k\text{-bimod}
\]
by taking
\[
C^\ast\ox^\pi M=(C^\ast\ox M, d_{C^\ast\ox M}-[\pi,-]).
\]
and $C^\ast\ox^\pi \phi=C^\ast\ox\phi$ for any dg bimodule map $\phi:M\to N$.
\end{definition}

Here we take $\pi\in C^\ast\ox A$ to be the preimage of $\pi$ along (\ref{naturalmapper}), by abuse of notation.  Note that when $(C,A)$ is sufficiently finite we have $\pi=\sum_i \pi_i^\ast\ox\pi_i$ for some elements $\pi_i^\ast\in C^\ast$ and $\pi_i\in A$.  So, even when $(C,M)$ is not locally finite, we can still consider the degree $1$ operation
\[
d_{C^\ast\ox M}-[\pi,-]=d_{C^\ast\ox M}-\sum_i[\pi^\ast_i\ox\pi_i,-]
\]
on the graded space $C^\ast\ox M$.
\par

The fact that $C^\ast\ox^\pi M$ is actually a chain complex follows from the fact that the differential $d_{\Hom^\pi_k}$ induces the given operation $d_{C^\ast\ox M}-[\pi,-]$ on $C^\ast\ox M$ by way of the embedding analogous to (\ref{naturalmapper}).  This observation also implies

\begin{proposition}\label{natmapprop}
Suppose $\pi:C\to A$ is a twisting cochain with $(C,A)$ sufficiently finite.  Then the natural chain maps
\begin{equation}\label{naturalmapper2}
C^\ast\ox^\pi M\to \Hom^\pi_k(C,M),\ \ f\ox m\mapsto (c\mapsto (-1)^{|c||m|}f(c)m)
\end{equation}
produce a natural transformation $C^\ast\ox^\pi-\to \Hom^\pi_k(C,-)$ which restricts to a natural isomorphism on the full subcategory of sufficiently finite bimodules in dg $A$-bimod.  Additionally, there is a natural transformation
\[
C^\ast\ox^\pi -\to \Hom_{A^e}(A\ox_\pi C\ox_\pi A,-),
\]
defined by composing (\ref{naturalmapper2}) with $\mathrm{res}^{-1}$, which is also an isomorphism on the subcategory of sufficiently finite dg bimodules.
\end{proposition}

\section{Presentation of the main theorem: an algebra identification $H^\bullet(\Hom^\pi_k(C,A))=HH^\bullet(A)$}
\label{cupprod}

For any twisting cochain $\pi:C\to A$ from a curved dg coalgebra to a dg algebra, we have a canonical map
\[
\varepsilon:A\ox_\pi C\ox_\pi A\to A,\ \ a\ox c\ox b\mapsto ab\epsilon_C(c)
\]
of dg bimodules.
\par

Recall that a coalgebra $C$ is called connected if it has a unique one dimensional simple subcoalgebra $C_0=k$ \cite[Definition 5.1.5]{M}.  We say a curved dg algebra $C=(C,d_C, f_C)$ is connected, or ``cocomplete", if it is connected as a coalgebra and $d_C|C_0\subset C_0$.  Compatibility of $d_C$ with the counit implies $d_C|C_0=0$.  In this case the standard coradical filtration
\[
\{F_nC:=\ker(C\overset{\Delta^{(n)}}\to C^{\ox n}\to (C/C_0)^{\ox n})\}_n
\]
will also satisfy $d_C(F_n C)\subset F_n C$.  The following definition is standard \cite[Definitions 2.2.1.1]{L}, \cite[Section 6.5]{positselski11}, \cite[Section 4.6]{K}.

\begin{definition}[Acyclic twisting cochain]
A twisting cochain $\pi:C\to A$ is called admissible if $C$ is a connected curved dg coalgebra and $\pi|C_0=0$.  A twisting cochain $\pi:C\to A$ is called acyclic if $\pi$ is admissible and the dg $A$-bimodule map
\[
\varepsilon:A\ox_\pi C\ox_\pi A\to A
\]
is a quasi-isomorphism.
\end{definition}

Our main example of an acyclic twisting cochain will be the canonical twisting cochain $\pi:(B^!)^\ast\to B$ associated to any filtered Koszul algebra $B$.  In this case the map $\varepsilon:B\ox_\pi (B^!)^\ast\ox_\pi B\to B$ will be the standard Koszul resolution (see Section \ref{kos}).  Before continuing we need to give a definition.

\begin{definition}[Hochschild cohomology of a dg algebra]
Let $A$ be any dg-algebra.  Then we take
\[
HH^\bullet(A,M)=H^\bullet(\RHom_{A^e}(A,M))\text{ and }HH^\bullet(A)=HH^\bullet(A,A).
\]
\end{definition}

For the uninitiated reader it may not even clear that we can derive the functor $\Hom_{A^e}(A,-)$ in general.  We will give a more explicit definition in Section \ref{rHH}, where more information on derived categories of dg modules and derived functors will be given.  For the time being, we give complete proofs of the following theorems only in the ring theoretic settings, i.e. the setting in which $A$ is concentrated in degree $0$ and $C$ is bounded above.  The completed proofs are given in Section \ref{rHH}.
\par

We begin with our main theorem about the Hochschild cohomology ring $HH^\bullet(A)$, then consider the general cohomologies $HH^\bullet(A,M)$.

\begin{theorem}\label{thm1}
Let $\pi:C\to A$ be any twisting cochain and take $K=A\ox_\pi C\ox_\pi A$.  The map $l:\Hom^\pi_k(C,A)\to \Hom_{A^e}(K,K)$ defined by
\begin{equation}\label{937}
f\mapsto \left(a\ox c\ox b\mapsto (-1)^{|f|(|a|+|c_1|)}a\ox c_1\ox f(c_2)b\right)
\end{equation}
is a map of dg algebras.  Furthermore, if $\pi$ is acyclic then the map $l$ is a quasi-isomorphism and we have an identification of graded rings $HH^\bullet(A)=H^\bullet(\Hom^\pi_k(C,A))$.
\end{theorem}

\begin{proof}[Proof in ring theoretic setting]
Take $K=A\ox_\pi C\ox_\pi A$.  Note that for any dg algebra $\Pi$, left dg module $M$, and generic elements $\sigma\in\Pi$ and $m\in M$, the formula $d_M(\sigma m)=d_\Pi(\sigma)m+(-1)^{|\sigma|}\sigma d_M(m)$ is exactly the statement that the left multiplication map $l_\Pi:\Pi\to \Hom_k(M,M)$ is a chain map.  Associativity of the action tells us that the left multiplication map is also an algebra, and hence dg algebra, map.  Therefore we get a dg algebra map
\[
l_{\Hom^\pi_k}:\Hom^\pi_k(C,A)\to \Hom_k(K,K)
\]
given by the formula (\ref{937}), since $K=A\ox_\pi C\ox_\pi A=(A\ox C\ox A)^\pi$ is a left $\Hom^\pi_k(C,A)$-module under the action given in Lemmas \ref{lem613} and \ref{lem627}.  We simply note that each map $l_{\Hom^\pi_k}(f)$ is left and right $A$-linear to see that the image of $l_{\Hom^\pi_k}$ lay in the dg subalgebra $\Hom_{A^e}(K,K)\subset \Hom_k(K,K)$.  This produces $l$ as the dg algebra map given by restricting the codomain of $l_{\Hom^\pi_k}$.
\par

In the case that $\pi$ is acyclic, $C$ is bounded above, and $A$ is concentrated in degree $0$, the complex $K$ provides a free bimodule resolution of $A$.  Whence the functor $\Hom_{A^e}(K,-)$ preserves quasi-isomorphisms.  In particular, the map
\[
\varepsilon_\ast:\Hom_{A^e}(K,K)\to \Hom_{A^e}(K,A)\underset{\mathrm{res}}\cong \Hom^\pi_k(C,A)
\]
will be a quasi-isomorphism.  (Recall that the restriction map is a chain isomorphism, by Proposition \ref{prop731}.)  Since $\epsilon(c_1)f(c_2)=f(\epsilon(c_1)c_2)=f(c)$ for each $c\in C$ we see that $\varepsilon_\ast l=id_{\Hom^\pi_k(C,A)}$.  Since $\varepsilon_\ast$ is a quasi-isomorphism this then implies that $l$ is also a quasi-isomorphism.
\end{proof}

\begin{corollary}\label{cor}
When $(C,A)$ is sufficiently finite, and $\pi:C\to A$ is an acyclic twisting cochain, then we have an identification of graded rings
\[
HH^\bullet(A)=H^\bullet(C^\ast\ox^\pi A)=H^\bullet(C^\ast\ox A,d_{C^\ast\ox A}-[\pi,-]).
\]
\end{corollary}

\begin{proof}
This is an immediate consequence of the fact that the isomorphism $C^\ast\ox^\pi A\cong \Hom^\pi_k(C,A)$ of Proposition \ref{natmapprop} is a dg algebra isomorphism.
\end{proof}

Let us now take a step back to consider the cohomologies $HH^\bullet(A,M)$.

\begin{theorem}\label{thmotherstuff}
If $\pi:C\to A$ is an acyclic twisting cochain then
\begin{enumerate}
\item for any dg bimodule $M$ we have $HH^\bullet(A,M)=H^\bullet\big(\Hom^\pi_k(C,M)\big)$.
\item We have an $H^\bullet(A)$-bimodule identification
\[
HH^\bullet(A,A^e)=H^\bullet\big(\Hom^\pi_k(C,A^e)\big),
\]
where the $A$-bimodule structure on functions in $\Hom^\pi_k(C,A^e)$ is induced by the inner bimodule structure on $A^e=A\ox A$.
\item If $(C,M)$ is sufficiently finite then $HH^\bullet(A,M)=H^\bullet(C^\ast\ox^\pi M)$ and if $(C,A^e)$ is sufficiently finite then the identification $HH^\bullet(A,A^e)=HH^\bullet(C^\ast\ox^\pi A^e)$ is one of $H^\bullet(A)$-bimodule.
\end{enumerate}
\end{theorem}

\begin{proof}[Proof in ring theoretic setting]
Take $K=A\ox_\pi C\ox_\pi A$.  Recall that, according to Proposition \ref{prop731}, the restriction map $\mathrm{res}_{A^e}:\Hom_{A^e}(K,M)\to \Hom^\pi_k(C,M)$ is an isomorphism of chain complexes for each $M$.  For (1) we need only know that $K=A\ox_\pi C\ox_\pi A$ is such that $\Hom_{A^e}(K,-)=\RHom_{A^e}(A,-)$.  As was explained in the partial proof to \ref{thm1}, this is clear when $A$ is concentrated in degree $0$ and $C$ is bounded above.  For (2) we simply note that the restriction map is an isomorphism of bimodules.  Part (3) is an immediate consequence of Proposition \ref{natmapprop} and the fact that the isomorphism (\ref{naturalmapper2}) is, again, an isomorphism of $A$-bimodules when $M=A^e$.
\end{proof}

\begin{remark}
Let $C^\bullet(A)=\Hom^{\mathrm{univ}}_k(\mathscr{B}A,A)$ denote the Hochschild cochain complex for $A$.  Theorem \ref{thm1} can alternately be proved by showing that the quasi-isomorphism $C^\bullet(A)\to \Hom_{A^e}(K,A)$ dual to the canonical embedding of $K$ into the bar resolution for $A$ (see \cite[Proposition 3.9]{Pr}) maps the cup product of elements in $C^\bullet(A)$ to the convolution product of their images.  This will be a more appropriate proof if one wishes to address the cup product on Hochschild cohomology with coefficients in some ring extension $A'$ of $A$.
\label{rmrkr}
\end{remark}
\begin{remark}
Since we already have a quasi-isomorphism at the level of dg algebras in Theorem~\ref{thm1}, one can verify that the identification $HH^\bullet(A)=H^\bullet(\Hom^\pi_k(C,A))$ is in fact one of $A_\infty$-algebras.
\end{remark}

\section{Hochschild cohomology of dg algebras: proofs of main theorems}
\label{rHH}

Here we give an an overview of some of the definitions and results from Barthel, May, and Riehl's paper \cite{barthelmayriehl}, and complete the proofs of Theorems \ref{thm1} and \ref{thmotherstuff}.  The paper \cite{barthelmayriehl} is concerned with analyzing a number of model structures on categories of dg modules.  We will, however, avoid discussing model categories at length.  Let us only say that since we are over a field, the (unbounded) derived category of $k$ is equal to the (unbounded) homotopy category of $k$.  This implies that the $q$-model structure and $r$-model structure from \cite{barthelmayriehl} are actually the same.  So we can use the authors results for the $r$-model structure to address the standard derived category of a dg algebra $\Pi$,
\[
D(\Pi)=\mathrm{Ho}(\mathrm{dg\ }\Pi\text{-mod})=\mathrm{dg}\ \Pi\text{-mod}[Quasi\text{-}isom^{-1}].
\]
\par

We fix a dg algebra $\Pi$ for the moment.

\begin{definition}[Semi-projective dg modules]
A dg $\Pi$-module $M$ is called semi-projective ($q$-semi-projective in \cite{barthelmayriehl}) if $M$ is projective as a $\Pi$-module, after forgetting the differential, and the hom complex $\Hom_\Pi(M,N)$ is acyclic whenever $N$ is acyclic.
\end{definition}

Since the construction of the mapping cone commutes with the hom complex functor, we see that $M$ is semi-projective if and only if it is projective as a (non-dg) $\Pi$-module and $\Hom_\Pi(M,-)$ preserves quasi-isomorphism.  So, in this case, $\Hom_\Pi(M,-)$ induces a functor on the localizations $D(\Pi)\to D(k)$.

\begin{definition}[Split filtrations]
A split filtration of a dg $\Pi$-module $M$ is a filtration $M=\cup_i F_i M$ with each $F_i M$ a dg submodule, $F_{-1}M=0$, and each quotient $F_iM/F_{i-1}M$ isomorphic to a dg bimodule of the form $\Pi\ox E$ for some $k$-complex $E$.
\end{definition}

We are principally interested in the following result.

\begin{proposition}[{\cite[Propositions 9.19]{barthelmayriehl}}]\label{prop1034}
Let $M$ be a dg $\Pi$-module.  If $M$ admits a split-filtration then $M$ is semi-projective.
\end{proposition}

We now turn our attention back to dg bimodules over a dg algebra $A$, in which case $\Pi=A^{e}$.

\begin{proposition}\label{prop1044}
If $\pi:C\to A$ is an admissible twisting cochain then the dg bimodule $A\ox_\pi C\ox_\pi A$ is semi-projective over $A^e$.
\end{proposition}

This result was inspired by \cite[Proposition 10.18]{barthelmayriehl}, which initiated my interest in the paper \cite{barthelmayriehl}.  The proof is also rather similar.

\begin{proof}
We filter by the coradical filtration on $C$.  More specifically, we take
\[
F_i=F_i(A\ox_\pi C\ox_\pi A):=A\ox F_i C\ox A.
\]
Since the differential on $C$ is filtered, as is the comultiplication, the differential on the twisted tensor complex does respect this filtration.  In fact, from the formula
\[
\ba{l}
d_{A\ox_\pi C\ox_\pi A} = d_{A\ox C\ox A}-[\pi,-]\\
= d_{A\ox C\ox A}+\big(\mu(id_A\ox\pi)\ox id_C\ox id_A-id_A\ox id_C\ox\mu(\pi\ox id)\big)\left(id_A\ox\Delta\ox id_A\right)
\ea
\]
for the differential on the twisted tensor product, and the fact that $\pi|C_0=0$, we see that the portion $[\pi,-]$ of the differential vanishes in the associated graded complex.  So we have
\[
F_i/F_{i-1}\cong A\ox (F_i C/F_{i-1} C)\ox A
\]
as an $A$-bimodule, where the differential is the product differential.  Whence, by Proposition \ref{prop1034} the complex is semi-projective.
\end{proof}

We will call a map $p:\tilde{M}\to M$ of dg $\Pi$-modules a semi-projective approximation of $M$ if $\tilde{M}$ is semi-projective and $p$ is a surjective quasi-isomorphism.  This notions correspond $p$ being an acyclic fibration.

\begin{lemma}\label{lem1083}
Let $\pi:C\to A$ be a twisting cochain.  If $\pi$ is acyclic then $\varepsilon:A\ox_\pi C\ox_\pi A\to A$ is a semi-projective approximation.
\end{lemma}

\begin{proof}
In light of the previous Proposition, we need only show that the map $\varepsilon$ is surjective.  However, this is clear since we have the $k$-section
\[
A\to A\ox A\cong A\ox C_0\ox A\to A\ox C\ox A,\ \ a\mapsto a\ox 1\ox 1.
\]
\end{proof}

There is now an obvious definition of the Hochschild cohomology, at least from the perspective of homological algebra and derived functors.

\begin{redefinition}[Hochschild cohomology of a dg algebra]\label{redefiner}
Let $A$ be a dg algebra.  We define the Hochschild cohomology $HH^\bullet(A)$ as the cohomology of the hom complex $\Hom_{A^e}(\tilde{A},\tilde{A})$, where $\tilde{A}\to A$ is a semi-projective approximation of $A$ over $A^e$.
\end{redefinition}

Obviously each of these hom complexes $\Hom_{A^e}(\tilde{A},\tilde{A})=\End_{A^e}(\tilde{A})$ will be a dg algebra.  So the Hochschild cohomology still admits a natural product.  This graded ring structure is well defined and independent of choice of resolution.
\par

We now complete the proof of Theorem \ref{thm1}.

\begin{proof}[Completed proof of Theorem \ref{thm1}]
By Lemma \ref{lem1083} $\varepsilon:K\to A$ will be a semi-projective approximation of $A$.  Thus the induced map
\[
\varepsilon_\ast:\Hom_{A^e}(K,K)\to \Hom_{A^e}(K,A)\cong\Hom_k^\pi(C,A)
\]
will be a quasi-isomorphism, and we can simply repeat the latter half of the proof of Theorem \ref{thm1} to get the desired result.
\end{proof}

We can also complete the proof of Theorem \ref{thmotherstuff}.

\begin{proof}[Completed proof of Theorem \ref{thmotherstuff}]
We simply note that $K$ is a semi-projective approximation, by Lemma \ref{lem1083}, so that $\RHom_{A^e}(A,-)=\Hom_{A^e}(K,-)\cong \Hom^\pi_k(C,-)$.  
\end{proof}

In closing, let us say a few words about the Hochschild cohomology of a dg algebra.  For any dg algebra $A$ we will always have the bar dg coalgebra $\mathscr{B}A$ and universal twisting cochain $\mathrm{univ}:\mathscr{B}A\to A$.  We then get the standard map
\[
\varepsilon:\mathscr{B}ar A=A\ox_{\mathrm{univ}}\mathscr{B}A\ox_{\mathrm{univ}}A\to A,
\]
which is a quasi-isomorphism since the mapping cone has a canonical contracting homotopy.  Then we get, by Theorem \ref{thm1}, that
\[
HH^\bullet(A)=H^\bullet(\Hom^{\mathrm{univ}}_k(\mathscr{B}A,A)).
\]
But $\Hom^{\mathrm{univ}}_k(\mathscr{B}A,A)$ is the standard Hochschild cochain complex.  So our derived functor version of the Hochschild cohomology is the same as the deformation theoretic Hochschild cohomology.  In particular, $\Sigma HH^{\bullet}(A)$ admits a graded Lie structure under which the solutions to the Maurer-Cartan equation correspond to infinitesimal deformations of $A$ (see, for example, \cite[Section 6.3]{KSdeformation}).

\section{Filtered Koszul rings and Koszul duals (following Positselski)}
\label{sec:koszul}

For the remainder of the paper $B$ will denote a (non-dg) algebra.  Internal gradings will be denoted with subscripts.  We reserve the superscript notation for (co)homological gradings.
 
\subsection{Graded Koszul duality with signs}
\label{kos0}

A Koszul algebra is a finitely generated connected graded algebra $B$, i.e.\ a graded algebra of the form
\[
B=k\oplus B_1\oplus B_2\oplus\cdots,
\]
such that $\Ext_B(k,k)$ is generated by $\Ext_B^1(k,k)$ as an algebra.  Here $k={_Bk}$ denotes the graded simple module $B/(B_{\geq 1})$.  The \emph{Koszul dual} of a Koszul algebra $B$ is the algebra $\Ext_B(k,k)$.  To avoid confusion with the filtered case, we denote the Koszul dual by $E$ for the moment.
\par

Any Koszul algebra will have a quadratic presentation $B=k\langle V\rangle/(R)$.  Let us fix a Koszul algebra with such a presentation.  Here $R\subset V\ox V$ is the subspace of quadratic relations for $B$.  It is well known that we have a presentation $E\cong k\langle V^\ast\rangle/(R^\perp)$.  We give here a description of the Koszul dual which takes into account the homological grading on the implicit Koszul resolution of $k$, which gives rise to the Koszul dual.
\par

We let $T\langle V\rangle=\oplus_{n\geq 0}V^{\ox n}$ denote the tensor coalgebra on $V$.  Recall that the comultiplication on $T\langle V\rangle$ is defined by ``separation of tensors"
\[
\mathbf{v}=(v_{1}\ox\dots\ox v_{n})\mapsto (1)\ox(\mathbf{v})+(\mathbf{v})\ox(1)+ \sum_{1\leq j\leq n-1} (v_{1}\ox\dots\ox v_j)\ox(v_{j+1}\ox\dots\ox v_{n}).
\] 
We consider $T\langle V\rangle$ to be homologically graded by taking $V$ to be in degree $-1$.  The following lemma is well known.  See for example \cite[Section 4.7]{K}, \cite[Sections 3.1.3--3.2.2]{LV}.

\begin{lemma}\label{intcoalg}
The graded subspace $W$ of $T\langle V\rangle$ defined by $W^0=k$, $W^{-1}=V$, and 
\begin{equation}
W^{-i}=\bigcap_{i_1+i_2=i-2} V^{\ox i_1}\ox R\ox V^{\ox i_2}
\label{int}
\end{equation}
for all $i\geq 2$, is a graded subcoalgebra of $T\langle V\rangle$.
\end{lemma}

It is a standard fact that there is an algebra isomorphism $E\cong W^\ast$, where $W^\ast$ is given the unsigned product $(f\star g)(c):=f(c_1)g(c_2)$.  For our purposes however, we will need an identification $E=W^\ast$ which employs our signed product on $W^\ast$.
\par

Consider $k\langle V^\ast\rangle$, the free algebra on the degree $1$ space $V^\ast$.  We have the canonical algebra isomorphism
\begin{equation}\label{free}
\ba{l}
k\langle V^\ast\rangle\to (T\langle V\rangle)^\ast\\
f_1\ox\dots\ox f_n\mapsto (v_1\ox\dots \ox v_n\mapsto (-1)^{n(n-1)/2}f_1(v_1)\dots f_n(v_n)).
\ea
\end{equation}
Here the $f_i$ are in $V^\ast$, the $v_i$ are in $V$, the function $f_1\ox\dots\ox f_n$ will vanish off $V^{\ox n}$, and the exponent $n(n-1)/2=\sum_{l=0}^{n-1}l$ comes from commuting the degree $-1$ variables $v_i$ past the degree $1$ maps $f_i$.  If we then compose the isomorphism (\ref{free}) with the dual of the inclusion $W\to T\langle V\rangle$ we get an algebra map $k\langle V^\ast\rangle\to W^\ast$.  The kernel of this map obviously contains the ideal generated by $R^\perp$, and it follows by the standard identification $E\cong W^\ast$ that the induced map $k\langle V^\ast\rangle/(R^\perp)=E\to W^\ast$ is an isomorphism.  This isomorphism simply sends a monomial $f_1\dots f_n$ in $E$ to the function
\[
f_1\dots f_n:W\to k,\ v_1\ox\dots \ox v_n\mapsto (-1)^{n(n-1)/2}f_1(v_1)\dots f_n(v_n).
\]
It is via this particular isomorphism that we identify $E$ with $W^\ast$ as a graded algebra.

\begin{remark}\label{rmrk0}
The sign conventions we employ here make no difference in the presentation of the Koszul dual as a graded ring $E=k\langle V^\ast\rangle/(R^\perp)=k\langle V^\ast\rangle/(\sum_{i,j}c^{ij}_\ell f_i\ox f_j)_\ell$.  The conventions do make a difference once we start considering differentials and curvature.
\end{remark}

\subsection{Filtered Koszul algebras and their Koszul dual (curved) dg algebras}
\label{afiltkos}

The class of algebras we will be interested in are the following.

\begin{definition}[Filtered Koszul algebras]
A $\mathbb{Z}_{\geq 0}$-filtered algebra $B=\cup_{i\geq 0} F_i B$ such that $\gr B$ is Koszul is called a filtered Koszul algebra.
\end{definition}

Let $B$ be a filtered Koszul algebra, and $\gr B=k\langle V\rangle/(R)$ be its (graded) Koszul associated graded algebra, with $R$ the space of degree $2$ relations.  One can check then that we will have a presentation 
\begin{equation}
\label{presenting}
B=k\langle V\rangle/(r+\alpha_1(r)+\alpha_0(r))_{r\in R}.
\end{equation}
for some linear functions $\alpha_1:R\to V$ and $\alpha_0:R\to k$.  These functions are not determined uniquely by $B$, but depend (uniquely) on a choice of section $V=F_1A/k\to F_1 A$ (See \cite[\S 2]{P}).
\par

Let us take $E$ to be the algebra of extensions $\Ext_{\mathrm{gr}B}(k,k)$ for the associated graded algebra $\gr B$, which we have assumed to be Koszul.  Recall that the algebra $E$ is given as the dual of the graded coalgebra $W=\cdots\oplus R\oplus V\oplus k$ of Lemma \ref{intcoalg}.  In \cite{P}, Positselski proves the following

\begin{proposition}[{\cite[Proposition 2.2]{P}}]\label{posprop}
Suppose $B$ is filtered Koszul with a presentation as in (\ref{presenting}).
\begin{enumerate}
\item The function $V^\ast=E^1\to R^\ast=E^2$ given by precomposition with $\alpha_1$, $f\mapsto f\alpha_1$, extends to a well defined graded degree $1$ derivation $d^B$ on $E$.
\item If we take
\[
c^B=-\alpha_0\in R^\ast=E^2,
\]
then the triple $(E,d^B,c^B)$ defines a curved dg algebra structure on the algebra of extensions $E=\Ext_{\gr B}(k,k)$ of the Koszul algebra $\gr B$.
\end{enumerate}
\end{proposition}

\begin{proof}
The proof is the same as in \cite{P}.  We only note here that the sign on the curvature has changed due to our signed identification with $W^\ast$ (see Remark \ref{rmrk0}).
\end{proof}

It is this structure which we view as the Koszul dual of $B$.  Here we could take $B$ to be a Weyl algebra or Clifford algebra, for example.

\begin{definition}[The Koszul dual]
Let $B$ be a filtered Koszul algebra.  The curved dg algebra $(\Ext_{\gr B}(k,k), d^B, c^B)$ described above will be called the Koszul dual curved dg algebra to $B$.  It will generally be denoted $B^!=(B^!,d_{B^!},c_{B^!})$.
\end{definition}

Note that when $\alpha_0=0$ in the presentation (\ref{presenting}), the algebra $B$ is augmented, the curvature $c_{B^!}$ vanishes, and the Koszul dual of $B$ is a non-curved dg algebra.  In this augmented case, Proposition \ref{posprop} already appears in Priddy's original work on Koszul resolutions \cite[Theorem 4.3]{Pr}.  In the case that $B$ is the universal enveloping algebra of a Lie algebra $\mathfrak{g}$, for example, the dg algebra $(E, d^B)$ is the Chevalley-Eilenberg dg algebra of $\mathfrak{g}$, $(E,d^B)=(\bigwedge^\bullet\mathfrak{g}^\ast,d_{CE})$.

\subsection{Example: The $n$th Weyl algebra}

In the case of the $n$th Weyl algebra
\[
A_n(k)=k\langle x_1,\dots, x_n,{\small\frac{\partial}{\partial x_1}},\dots,{\small\frac{\partial}{\partial x_n}}\rangle/([x_j,x_i],[\frac{\partial}{\partial x_j},\frac{\partial}{\partial x_i}],[{\small\frac{\partial}{\partial x_j}},x_i]-\delta_{ji})_{ij},
\]
we have $\gr A=k[x_1,\dots, x_n,\frac{\partial}{\partial x_1},\dots,\frac{\partial}{\partial x_n}]$ and 
\[
\Ext_{\gr A}(k,k)=k[\lambda_1,\dots,\lambda_n,\theta_1,\dots,\theta_n]=:k[\Lambda,\Theta].
\]
In the second algebra, the variables $\lambda_i$ and $\theta_j$ are the duals of the $x_i$ and $\frac{\partial}{\partial x_j}$ respectively.  We consider these functions to have homological degree $1$, and the algebra $k[\Lambda,\Theta]$ is the free {\it graded} commutative algebra with these generators.  (So the variables anti-commute.)
\par
Recall our identification of $\Ext_{\gr A}^2(k,k)=(V^\ast\ox V^\ast)/R^\perp$ with $R^\ast$ is given by sending a monomial $f_1f_2$ to the function $\sum_i r_i\ox r'_i\mapsto -\sum f_1(r_i)f_2(r_i')$.  So $\lambda_i\theta_j$ gets identified with the function $R\to k$ defined on basis elements by
\[
\lambda_i\theta_j: \ba{l} x_k\ox x_l-x_l\ox x_k\mapsto 0\\
\frac{\partial}{\partial x_k}\ox x_l-x_l\ox \frac{\partial}{\partial x_k}\mapsto \delta_{il}\delta_{jk}\\
\frac{\partial}{\partial x_k}\ox \frac{\partial}{\partial x_l}- \frac{\partial}{\partial x_l}\ox \frac{\partial}{\partial x_k}\mapsto 0\ea
\]
Whence, in this case, the curvature element $c^{A_n(k)}=-\alpha_0$ in $k[\Lambda,\Theta]^2$ will be the sum
\[
c^{A_n(k)}=\sum_{i=1}^n\lambda_i\theta_i.
\]
The corresponding Koszul dual of $A_n(k)$ will be the curved dg algebra
\[
(k[\lambda_1,\dots,\lambda_n,\theta_1,\dots,\theta_n], 0, c^{A_n(k)}).
\]

\subsection{Example: PBW deformations of skew polynomial rings}
\label{skewp}

Take $V=\langle x_1,\dots,x_n\rangle$ and let $S_Q(V)$ denote the skew polynomial ring
\[
S_Q(V)=k\langle x_1,\dots, x_n\rangle/(x_jx_k-q_{jk}x_kx_j)
\]
for $Q=[q_{jk}]$ a multiplicatively skew symmetric matrix ($q_{jk}=q_{kj}^{-1}$) with $q_{jj}=1$.  The Koszul dual of the skew polynomial ring in the skew exterior algebra
\[
\Ext_{S_Q(V)}(k,k)={\bigwedge}_QV^\ast:=k\langle \lambda_1,\dots,\lambda_n\rangle/
(\lambda_k\lambda_j+q_{jk}\lambda_j\lambda_k, \lambda_j^2).
\]
An augmented PBW deformation $B$ of $S_Q(V)$ will be given by some constants $c_i^{jk}$ so that the relations on our PBW deformation $B$ will be given by
\[
B=k\langle x_1,\dots,x_n\rangle/(x_jx_k-q_{jk}x_kx_j-\sum_i c^{jk}_i x_i).
\]
So $\alpha_1:R\to V$ will be the function $x_j\ox x_k-q_{jk}x_k\ox x_j\mapsto -\sum_i c_i^{jk}x_i$.
\par
Now, on the Koszul dual, the product $\lambda_i\lambda_j$ in $({\bigwedge}_QV^\ast)^2$ is identified with the function
\[
x_l\ox x_m -q_{lm}x_m\ox x_l\mapsto-(\lambda_i(x_l)\lambda_j(x_m)-q_{lm}\lambda_i(x_m)\lambda_j(x_l))=-\delta_{il}\delta_{jm}+q_{lm}\delta_{im}\delta_{jl},
\]
i.e.\ the negated dual of the relations $x_ix_j-q_{ij}x_jx_i$.  So
\[
\ba{rl}
d_{B^!}(\lambda_i)&=\lambda_i\alpha_1\\
&=(x_j\ox x_k-q_{jk}x_k\ox x_j\mapsto -\sum_l c_l^{jk}\lambda_i(x_l))\\
&=(x_j\ox x_k-q_{jk}x_k\ox x_j\mapsto -c_i^{jk})\\
&=\sum_{j<k} c_i^{jk}\lambda_j\lambda_k.
\ea
\]
We will come back to this example in the next section.

\section{Hochschild cohomology via $B^!$ for Koszul rings}
\label{kos}

\subsection{Koszul resolutions via twisting cochains} In this section we give a presentation of Koszul resolutions based on the work of Keller and Lef\'{e}vre-Hasegawa.  The original presentation, in the case that $B$ is graded Koszul, appears in \cite[Section 4.7]{K} and \cite{K1}.
\par

For the section we fix $B$ to be a filtered Koszul algebra with a presentation
\[
B=k\langle V\rangle/(r+\alpha_1(r)+\alpha_0(r))_{r\in R}
\]
and $B^!=(B^!,d_{B^!},c_{B^!})$ its Koszul dual (curved) dg algebra of Proposition \ref{posprop}.  Let $\gr B=k\langle V\rangle/(R)$ so that $B^!$ is dual to the graded coalgebra $W=\cdots\oplus (R\ox V\cap V\ox R)\oplus R\oplus V\oplus k$ of Lemma \ref{intcoalg}.  Positselski's proof of \cite[Proposition 2.2]{P}, in particular equation (2) of the proof given therein, implies
\par

\begin{lemma}[\cite{P}]\label{curveddgc}
There is a unique curved dg structure on $W$ given by $f_W=\alpha_0$ and $d_W|W^{-2}=\alpha_1$ such that the identification $B^!=W^\ast$ is one of curved dg algebras.
\end{lemma}

\begin{proof}
The proof of \cite[Proposition 2.2]{P} shows that such a curved dg coalgebra structure on $W$ exists.  The fact that $B^!=W^\ast$ as curved dg algebras is immediate from the definition of the curved dg algebra structure on the dual of a curved dg coalgebra given in Section \ref{sec:curved_stuffs}.
\end{proof}

\begin{definition}[Koszul dual coalgebra]
Given a filtered Koszul algebra $B$, with $\gr B= k\langle V\rangle/(R)$, the Koszul dual (curved) dg coalgebra to $B$ will be the, possibly curved, dg coalgebra $(W, d_W, f_W)$ of Lemma \ref{curveddgc}.  We will often write simply $W$ for $(W,d_W,f_W)$.
\end{definition}

\begin{lem/def}[The twisting cochain $e$]\label{lemdef}
Let $B$ be filtered Koszul and $W$ be its Koszul dual (curved) dg coalgebra.  Let $e:W\to B$ be the composition of the projection $W\to W^{-1}=V$ with the inclusion $V\to B,\ v\mapsto v$.  The map $e:W\to B$ is an acyclic twisting cochain.
\end{lem/def}

\begin{proof}
We need to verify the formula
\[
-f_W-e d_W+\mu(e\ox e)\Delta=0.
\]
It suffices to check that the above equation holds when evaluated at a homogeneous degree $-2$ element in $W$, since the left hand side vanishes on elements of all other degrees.  Recalling that $W^{-2}=R$, we evaluate on a relation $r=\sum_i r_i\ox r'_i$ to get
\[
\ba{l}
(-f_W-e d_W+\mu(e\ox e)\Delta)(r)\\
=-f_W(r)-e d_W(r)+\mu(e(r)\ox e(1)+e(1)\ox e(r)-\sum_i e(r_i)\ox e(r'_i))\\
=-\alpha_0(r)- e(\alpha_1(r))-\sum_i r_ir_i'\\
=-\alpha_0(r)-\alpha_1(r)+\alpha_1(r)+\alpha_0(r)\\
=0.
\ea
\]
\par

The fact that $W$ is connected is clear, since $W^0=k$.  Now, in the case where $B$ is graded Koszul the above complex $B\ox_e W\ox_e B$, along with the map $\varepsilon:B\ox_e W\ox_e B\to B$ is easily seen to recover the standard Koszul resolution \cite[proof of Proposition 3.3]{V}, \cite[Section 4.7]{K1}.  In general, one can employ the filtration 
\[
F_i(B\ox_e W\ox_e B)=\sum_{i_1+i_2+i_3=i} F_{i_1}B\ox W^{-i_2}\ox F_{i_3}B,
\]
which has associated graded complex equal to the Koszul resolution of $\gr B$, and an easy spectral sequence argument to see that $H^{<0}(B\ox_e W\ox_e B)=0$.  The fact that $H^{0}(B\ox_e W\ox_e B)=B$ is apparent.
\end{proof}

In following the standard terminology, we call the resolution $\varepsilon:B\ox_e W\ox_e B\overset{\sim}\to B$ the {\it Koszul resolution} of a filtered Koszul ring $B$.

\subsection{Hochschild cohomology via $B^!$}

As a consequence of Lemma \ref{lemdef}, we find that Corollary \ref{cor} specializes to the Koszul case to give

\begin{corollary}\label{thm}
For $B$ filtered Koszul with Koszul dual (curved) dg algebra $B^!$, we have an identification of ($A_\infty$-)algebras $HH^\bullet(B)=H^\bullet(B^!\ox^e B)$.
\end{corollary}

Recall that the object $B^!\ox^e B$ is the dg algebra 
\[
B^!\ox^e B=\big(B^!\ox B, d_{B^!\ox B}-[e,-]\big).
\]
If we suppose that $\gr B=k\langle V\rangle/(R)$, then the element $e$ is just the standard identity element in $(B^!\ox B)^1=V^\ast\ox V$.  Indeed, given a basis $\{x_i\}_i$ of $V$, and dual basis $\{\lambda_i\}_i$ of $V^\ast$, one can check that $e$ is the element $\sum_i \lambda_i\ox x_i$.  This element is called the ``identity element" because it is identified with the identity map under the natural isomorphism $V^\ast\ox V\overset{\cong}\to \Hom_k(V,V)$.  Whence we have
\[
HH^\bullet(B)=H^\bullet(B^!\ox^e B)=H^\bullet\big(B^!\ox B,d_{B^!\ox B}-[\sum_i \lambda_i\ox x_i,-]\big).
\]
\par

For the cohomologies $HH^\bullet(B,M)$, we have the following immediate corollary to Theorem \ref{thmotherstuff}.  The second portion of the statement provides a slight generalization of \cite[Theorem 9.1]{V2} to allow for filtered, not just graded, Koszul algebras.  One could also consider this result in relation to Yekutieli's computations of rigid dualizing complexes for universal enveloping algebras \cite{yekutieli00}.
 
\begin{corollary}\label{prop}
Given a filtered Koszul algebra $B$, and $B$ bimodule $M$ we have $H^\bullet(B^!\ox^e M)=HH^\bullet(B,M)$.  When $M={_BB\ox B_B}$ then the $H^\bullet(B^!\ox^e (B\ox B))=H^\bullet(B,B\ox B)$ as a bimodule.
\end{corollary}

\begin{remark}
It seems as though the most readily generalizable result is the identification
\[
HH^\bullet(B)=H^\bullet(\Hom^e_k(W,B)).
\]
For example, if one is interested in moving away from the (graded) Koszul case to the general case of connected graded algebras, one should replace the Koszul dual algebra $B^!$ with the Koszul dual $A_\infty$-algebra (see \cite{LZ}).  Taking the dual $B^!$ will give an $A_\infty$-coalgebra $W$ which will be connected to $B$ via an $A_\infty$-twisting cochain $e:W\to B$ (similar to the \cite[Section 4.4]{K}).  It may then be the case that we still have that the cohomology of the twisted homs $\Hom^e_k(W,B)$ is the Hochschild cohomology algebra.
\end{remark}

\bibliographystyle{abbrv}

\begin{thebibliography}{10}

\bibitem{barthelmayriehl}
T.~Barthel, J.~May, and E.~Riehl.
\newblock Six model structures for dg-modules over dgas: model category theory
  in homological action.
\newblock {\em New York J. Math.}, 20:1077--1159, 2014.

\bibitem{brown59}
E.~H. Brown.
\newblock Twisted tensor products, i.
\newblock {\em Ann. of Math.}, pages 223--246, 1959.

\bibitem{BS}
R.~O. Buchweitz, E.~L. Green, N.~Snashall, and {\O}.~Solberg.
\newblock Multiplicative structures for koszul algebras.
\newblock {\em Q. J. Math.}, 59(4):441--454, 2008.

\bibitem{chen77}
K.-T. Chen.
\newblock Iterated path integrals.
\newblock {\em Bull. Amer. Math. Soc.}, 83(5):831--879, 1977.

\bibitem{felix04}
Y.~Felix, J.-C. Thomas, and M.~Vigu{\'e}-Poirrier.
\newblock The hochschild cohomology of a closed manifold.
\newblock {\em Publications Math{\'e}matiques de l'IH{\'E}S}, 99:235--252,
  2004.

\bibitem{gerstenhaber63}
M.~Gerstenhaber.
\newblock The cohomology structure of an associative ring.
\newblock {\em Ann. of Math.}, pages 267--288, 1963.

\bibitem{gerstenhaber64}
M.~Gerstenhaber.
\newblock On the deformation of rings and algebras.
\newblock {\em Ann. of Math.}, pages 59--103, 1964.

\bibitem{grimleynguyenwitherspoon}
L.~Grimley, V.~C. Nguyen, and S.~Witherspoon.
\newblock Gerstenhaber brackets on {H}ochschild cohomology of twisted tensor
  products.
\newblock {\em preprint
  \href{http://arxiv.org/abs/1503.03531}{\rm{arXiv:1503.03531}}}.

\bibitem{gugenheim82}
V.~Gugenheim.
\newblock On a perturbation theory for the homology of the loop-space.
\newblock {\em J. Pure Appl. Algebra}, 25(2):197--205, 1982.

\bibitem{gugenheim_etal90}
V.~Gugenheim, L.~Lambe, J.~Stasheff, et~al.
\newblock Algebraic aspects of {C}hen's twisting cochain.
\newblock {\em Illinois J. Math.}, 34(2):485--502, 1990.

\bibitem{husemoller74}
D.~Husemoller, J.~C. Moore, and J.~Stasheff.
\newblock Differential homological algebra and homogeneous spaces.
\newblock {\em J. Pure Appl. Algebra}, 5(2):113--185, 1974.

\bibitem{K3}
B.~Keller.
\newblock A-infinity algebras in representation theory.
\newblock {\em online
  \href{webusers.imj-prg.fr/~bernhard.keller/publ/art.ps}{webusers.imj-prg.fr/~bernhard.keller/publ/art.ps}}.

\bibitem{keller03}
B.~Keller.
\newblock Derived invariance of higher structures on the {H}ochschild complex.
\newblock {\em preprint at
  \href{http://webusers.imj-prg.fr/~bernhard.keller/publ/dih.pdf}{webusers.imj-prg.fr/~bernhard.keller/publ/dih.pdf}}.

\bibitem{K2}
B.~Keller.
\newblock Introduction to {$A$}-infinity algebras and modules.
\newblock {\em Homology Homotopy Appl.}, 3(1):1--35, 2001.

\bibitem{K1}
B.~Keller.
\newblock Koszul duality and coderived categories (after {K}. {L}efevre).
\newblock \url{http://www.math.jussieu.fr/~keller/publ/kdc.pdf}, 2003.

\bibitem{K}
B.~Keller.
\newblock {$A$}-infinity algebras, modules and functor categories.
\newblock In {\em Trends in representation theory of algebras and related
  topics}, volume 406 of {\em Contemp. Math.}, pages 67--93. Amer. Math. Soc.,
  Providence, RI, 2006.

\bibitem{KSdeformation}
M.~Kontsevich and Y.~Soibelman.
\newblock {\em Deformation theory}.
\newblock Livre en pr{\'e}paration.

\bibitem{L}
K.~Lef\'{e}vre-Hasegawa.
\newblock {\em Sur les $A_\infty$-cat\'{e}gories Th`ese de doctorat,
  Universit´e Denis Diderot -- Paris 7, November 2003, available at B.
  Keller's homepage.}
\newblock PhD thesis, Universite Denis Diderot Paris 7, November 2003.

\bibitem{LV}
J.-L. Loday and B.~Vallette.
\newblock {\em Algebraic operads}, volume 346 of {\em Grundlehren der
  Mathematischen Wissenschaften [Fundamental Principles of Mathematical
  Sciences]}.
\newblock Springer, Heidelberg, 2012.

\bibitem{LZ}
D.-M. Lu, J.~H. Palmieri, Q.-S. Wu, and J.~J. Zhang.
\newblock {$A_\infty$}-algebras for ring theorists.
\newblock In {\em Proceedings of the {I}nternational {C}onference on
  {A}lgebra}, volume~11, pages 91--128, 2004.

\bibitem{manetti09}
M.~Manetti.
\newblock Differential graded lie algebras and formal deformation theory.
\newblock In {\em Algebraic geometry---Seattle 2005. Part 2}, pages 785--810,
  2009.

\bibitem{menichi01}
L.~Menichi.
\newblock The cohomology ring of free loop spaces.
\newblock {\em Homology Homotopy Appl.}, 3(1):193--224, 2001.

\bibitem{M}
S.~Montgomery.
\newblock {\em Hopf algebras and their actions on rings}, volume~82 of {\em
  CBMS Regional Conference Series in Mathematics}.
\newblock Published for the Conference Board of the Mathematical Sciences,
  Washington, DC, 1993.

\bibitem{negronthesis}
C.~Negron.
\newblock {\em Alternate Approaches to the Cup Product and Gerstenhaber Bracket
  on Hochschild Cohomology}.
\newblock PhD thesis, University of Washington, 2015.

\bibitem{negronwitherspoon15}
C.~Negron and S.~Witherspoon.
\newblock An alternate approach to the {L}ie bracket on {H}ochschild
  cohomology.
\newblock {\em to appear in Homology Homotopy Appl.}

\bibitem{nicolas08}
P.~Nicol{\'a}s.
\newblock The bar derived category of a curved dg algebra.
\newblock {\em J. Pure Appl. Algebra}, 212(12):2633--2659, 2008.

\bibitem{positselski11}
L.~Positselski.
\newblock {\em Two kinds of derived categories, {K}oszul duality, and
  comodule-contramodule correspondence}.
\newblock American Mathematical Society, 2011.

\bibitem{P}
L.~E. Positsel{\cprime}ski{\u\i}.
\newblock Nonhomogeneous quadratic duality and curvature.
\newblock {\em Funktsional. Anal. i Prilozhen.}, 27(3):57--66, 96, 1993.

\bibitem{Pr}
S.~Priddy.
\newblock Koszul resolutions.
\newblock {\em Trans. Amer. Math. Soc.}, 152:39--60, 1970.

\bibitem{SS04}
N.~Snashall and {\O}.~Solberg.
\newblock Support varieties and {H}ochschild cohomology rings.
\newblock {\em Proc. Lond. Math. Soc.}, 88(3):705--732, 2004.

\bibitem{V}
M.~Van~den Bergh.
\newblock Noncommutative homology of some three-dimensional quantum spaces.
\newblock In {\em Proceedings of {C}onference on {A}lgebraic {G}eometry and
  {R}ing {T}heory in honor of {M}ichael {A}rtin, {P}art {III} ({A}ntwerp,
  1992)}, volume~8, pages 213--230, 1994.

\bibitem{V2}
M.~Van~den Bergh.
\newblock Existence theorems for dualizing complexes over non-commutative
  graded and filtered rings.
\newblock {\em J. Algebra}, 195(2):662--679, 1997.

\bibitem{yekutieli00}
A.~Yekutieli.
\newblock The rigid dualizing complex of a universal enveloping algebra.
\newblock {\em J. Pure Appl. Algebra}, 150(1):85--93, 2000.

\end{thebibliography}

\def\cprime{$'$}

\end{document}